\documentclass[a4paper, american, reqno]{amsart}
\usepackage[foot]{amsaddr}

\usepackage[utf8]{inputenc}
\usepackage[T1]{fontenc}            
\usepackage{amsmath, amssymb}
\usepackage{amsthm, amsfonts}
\usepackage{algorithm}
\usepackage{algpseudocode}
\usepackage{dsfont}                 
\usepackage{url}                    
\usepackage{booktabs}               
\usepackage{longtable}              
\usepackage{caption}
\usepackage{enumitem}
\usepackage{pifont}
\usepackage[flushleft]{threeparttable}
\usepackage{graphicx}
\DeclareOldFontCommand{\rm}{\normalfont\rmfamily}{\mathrm}
\usepackage{natbib}
    \bibpunct[, ]{(}{)}{,}{a}{}{,}

\usepackage{xcolor}
\definecolor{darkgreen}{rgb}{0.31, 0.47, 0.26}

\usepackage{hyperref}
\hypersetup{
	colorlinks,
	linkcolor={blue!50!black},
	citecolor={darkgreen!75!black},
	urlcolor={blue},
	breaklinks=true,   
	colorlinks=true,   
	pdfusetitle=true,  
}

\usepackage[noabbrev, capitalise, nameinlink,]{cleveref}

\theoremstyle{plain}
\newtheorem{theorem}{Theorem}[section]
\newtheorem{proposition}[theorem]{Proposition}

\theoremstyle{definition}

\newtheorem{remark}[theorem]{Remark}

\theoremstyle{plain}
\newtheorem{result}{Result} 

\usepackage{pgfplots}
\usepackage{pgfplotstable}
\usetikzlibrary{positioning, calc}
\usepgfplotslibrary{groupplots} 
\usepgfplotslibrary{fillbetween}
\pgfplotsset{compat=1.16}
\pgfplotsset{grid = major, grid style={gray!30!white}}

\makeatletter
\patchcmd{\@settitle}{\uppercasenonmath\@title}{\scshape\large}{}{}
\patchcmd{\@setauthors}{\MakeUppercase}{\scshape\normalsize}{}{}
\makeatother
\AddToHook{cmd/@setemails/before}{\raggedright}

\usepackage{etoolbox}

\makeatletter

\pretocmd{\NAT@citex}{%
  \let\NAT@hyper@\NAT@hyper@citex
  \def\NAT@postnote{#2}%
  \setcounter{NAT@total@cites}{0}%
  \setcounter{NAT@count@cites}{0}%
    \forcsvlist{\stepcounter{NAT@total@cites}\@gobble}{#3}}{}{}
\newcounter{NAT@total@cites}
\newcounter{NAT@count@cites}
\def\NAT@postnote{}

\def\NAT@hyper@citex#1{%
  \stepcounter{NAT@count@cites}%
  \hyper@natlinkstart{\@citeb\@extra@b@citeb}#1%
  \ifnumequal{\value{NAT@count@cites}}{\value{NAT@total@cites}}
    {\ifNAT@swa\else\if*\NAT@postnote*\else%
     \NAT@cmt\NAT@postnote\global\def\NAT@postnote{}\fi\fi}{}%
  \ifNAT@swa\else\if\relax\NAT@date\relax
  \else\NAT@@close\global\let\NAT@nm\@empty\fi\fi
  \hyper@natlinkend}
\renewcommand\hyper@natlinkbreak[2]{#1}

\patchcmd{\NAT@citex}
  {\ifNAT@swa\else\if*#2*\else\NAT@cmt#2\fi
   \if\relax\NAT@date\relax\else\NAT@@close\fi\fi}{}{}{}
\patchcmd{\NAT@citex}
  {\if\relax\NAT@date\relax\NAT@def@citea\else\NAT@def@citea@close\fi}
  {\if\relax\NAT@date\relax\NAT@def@citea\else\NAT@def@citea@space\fi}{}{}

\makeatother

\title[The Differentiable Feasibility Pump]{The Differentiable Feasibility Pump}
\author[M. Cacciola, A. Forel]{Matteo Cacciola, Alexandre Forel}
\address[M. Cacciola, A. Forel]{Polytechnique Montréal, Montréal, QC, Canada}
\email{mvazcacciola@gmail.com, alexandre.forel@polymtl.ca}

\author[A. Frangioni]{Antonio Frangioni}
\address[A. Frangioni]{Dipartimento di Informatica, Universita di Pisa, Pisa, Italy}
\email{frangio@di.unipi.it}

\author[A. Lodi]{Andrea Lodi}
\address[A. Lodi]{Jacobs Technion-Cornell Institute, Cornell Tech and Technion–IIT, New York, NY, USA}
\email{andrea.lodi@cornell.edu}
\date{\today}

\newcommand{\cross}{\textcolor{gray}{\ding{55}}}
\newcommand{\tick}{\ding{52}}

\newcommand{\argmin}{\operatornamewithlimits{argmin}}

\newcommand{\round}[1]{\ensuremath{\left\lfloor#1\right\rceil}}
\DeclareMathOperator{\Exp}{\mathbb{E}}
\DeclareMathOperator{\Prob}{\mathbb{P}}
\DeclareMathOperator{\indicator}{\mathbb{I}}
\DeclareMathOperator*{\relu}{ReLU}

\def\nMax{N_{\text{max}}}
\def\reals{\mathbb{R}}
\DeclareMathOperator*{\loss}{\mathcal{L}}
\newcommand{\fcarg}[1]{\def\fc@rg{#1}\ifx\fc@rg\empty\fcdot\else\fc@rg\fi}

\begin{document}

\begin{abstract}
    Although nearly 20 years have passed since its conception, the feasibility pump algorithm remains a widely used heuristic to find feasible primal solutions to mixed-integer linear problems. Many extensions of the initial algorithm have been proposed. Yet, its core algorithm remains centered around two key steps: solving the linear relaxation of the original problem to obtain a solution that respects the constraints, and rounding it to obtain an integer solution. This paper shows that the traditional feasibility pump and many of its follow-ups can be seen as gradient-descent algorithms with specific parameters. A central aspect of this reinterpretation is observing that the traditional algorithm differentiates the solution of the linear relaxation with respect to its cost. This reinterpretation opens many opportunities for improving the performance of the original algorithm. We study how to modify the gradient-update step as well as extending its loss function. We perform extensive experiments on MIPLIB instances and show that these modifications can substantially reduce the number of iterations needed to find a solution.
\end{abstract}

\maketitle

\section{Introduction}
\label{sec:intro}

Despite significant improvements in open-source and commercial solvers in recent years, solving large-scale mixed-integer linear optimization problems remains a computational challenge. Primal heuristics, which return a solution that satisfies all constraints of the problem, are a cornerstone of the overall solving process. Among them, the feasibility pump, originally proposed by \citet{Fischetti2005feasibility}, remains a popular and efficient method.

The feasibility pump is based on two main operations: (i)~solving the linear relaxation of the original problem to obtain a solution that is not necessarily integer but respects the linear constraints, and (ii)~rounding the solution to respect the integrality constraints, at the risk of violating the linear constraints. The algorithm iterates over these two steps, changing the objective of the problem at each iteration. When it gets stuck in a cycle, restart operations are performed to perturb the objective and resume the algorithm.

In this paper, we propose a novel interpretation of the feasibility pump. We show that the traditional algorithm can be seen as a gradient-descent with a specific loss function, gradient estimation, and gradient update steps. This perspective builds upon the recent works on differentiating through optimization models and integrating them as layers into deep learning \citep{Amos2019differentiable, Amos2017optnet}. Differentiable optimization has seen extensive applications in structured learning and contextual optimization \citep{Kotary2021end, Mandi2024survey, Sadana2023survey}.

This reinterpretation allows us to derive a new generalized feasibility pump. Because of the central role of differentiable optimization, we call this generalization the \textit{differentiable feasibility pump}. It extends the original algorithm in multiple ways. First, it allows simple modifications of the gradient-update parameters as well as applying more advanced gradient-descent algorithms. Second, it allows adapting the loss function to take advantage of the gradient-descent setting. Specifically, we show that a piecewise convex loss function leads to better performance than the original one, which is piecewise linear. Third, it allows extending the loss function by adding differentiable loss terms. In this paper, we introduce a loss term to penalize the infeasibility of the rounded solution.

\subsection{Related works.} Several extensions have been proposed since the work of \citet{Fischetti2005feasibility}. While the original algorithm was proposed for binary mixed-integer linear problems, it was quickly extended to general linear mixed-integer problems \citep{Bertacco2007feasibility} and non-linear ones \citep[see, e.g.,][]{Bonami2009feasibility, DAmbrosio2012storm, Misener2012global, Bernal2020improving}. We present below several extensions and refer an interested reader to \citet{Berthold2019ten} and \citet{Berthold2025heuristicbook} for an extensive overview.

\textbf{Generalizations.} \citet{Boland2012new} interpret the feasibility pump as a proximal point algorithm. \citet{DeSantis2014feasibility} and \citet{DeSantis2013new} cast feasibility pumps as a Frank-Wolfe method and introduce concave penalty functions for measuring solution integrality. \citet{Geissler2017penalty} frame the feasibility pump as an alternating direction method applied to a specific reformulation of the original problem. They then study modifications of the restarts to escape cycles. \citet{Dey2018improving} focus on the restart components and develop extensions based on the WalkSAT local-search algorithm.

\textbf{Objective-informed.} A limitation of the original algorithm is that it has no explicit component to emphasize the quality of the solution in terms of cost. Ideally, a primal heuristic would return a feasible solution that also has a small cost for minimization problems. \citet{Achterberg2007improving} introduce an objective component to guide the pump towards feasible solutions with a low cost. Recently, \citet{Mexi2023using} introduce objective scaling terms.

\textbf{Constraint-informed.} A second often-discussed limitation is that the two key steps of the feasibility pump are performed ``blind'' to each other. That is, the integrality constraints are ignored when solving the linear relaxation, and the linear constraints are ignored when rounding the relaxed solution. Several approaches have been proposed to remedy this issue, mostly focusing on making the rounding component aware of the problem's constraints. \citet{Fischetti2009feasibility} investigate a modified rounding operation to encourage constraint propagation. \citet{Baena2011using} modify the rounding step to round on the segment between the computed feasible point and the analytic center of the relaxed problem. Other modifications include introducing an integer line search \citep{Boland2014boosting}, withs a barrier term to take into account active constraints \citep{Frangioni2021constraints}, or multiple reference vectors \citep{Mexi2023using}.

\subsection{Contributions.} In this paper, we make the following contributions:
\begin{enumerate}[label=(\roman*)]
    \item We reinterpret the feasibility pump algorithm as a gradient-descent with specific loss function, gradient estimation, and gradient-update parameters.
    \item We present the differentiable feasibility pump, a generalized feasibility pump algorithm based on gradient descent, and show how this general form encompasses many existing algorithms.
    \item We investigate several variations of the differentiable feasibility pump. In particular, we introduce a differentiable rounding operation that allows us to include a loss term to measure the infeasibility of the rounded solution. This provides a new solution to the ``blindness'' of the feasibility pump.
    \item We show the value of the differentiable pump in extensive experiments on MIPLIB instances. Our experiments show that the original feasibility pump is remarkably stable with respect to its hyperparameters. However, introducing a feasibility loss together with changes in the hyper-parameters can significantly improve performance. Another interesting result is that a differentiable pump using only feasibility loss requires one order of magnitude fewer restart operations, leading to a significantly less noisy algorithm.
\end{enumerate}

\subsection{Outline.}
\cref{sec:background} recalls the original feasibility pump and introduces background literature on differentiable optimization. \cref{sec:generalized} presents our reinterpretation of the feasibility pump as a gradient-descent algorithm and shows the equivalence with the original algorithm under a specific set of parameters. \cref{sec:extension} presents the differentiable feasibility pump and introduces a loss measuring the infeasibility of the rounded solution. \cref{sec:numerical} investigates experimentally the performance of several of these extensions. Finally, \cref{sec:conclusion} summarizes our findings and highlights directions for future research.
\section{Background}\label{sec:background}

We consider binary mixed-integer linear optimization problems of the form
\begin{subequations}
	\label{opt:binlp}
	\begin{alignat}{2}
		\min        & \quad && c^\top x, \label{opt:binlp:obj}\\
		\text{s.t.} &       && Ax \ge b,  \label{opt:binlp:const}\\
		            &       && x \in \left\lbrace 0, 1 \right\rbrace^n, \label{opt:binlp:domain}
	\end{alignat}
\end{subequations}
where $A \in \reals^{m \times n}$, $b \in \reals^{m}$, and $c \in \reals^n$ is the cost vector. We denote the feasible region of Problem~\eqref{opt:binlp} as $\mathcal{X} = \left\{ x~\in~\{0,1\}^n ~|~ \allowbreak Ax~\ge~b \right\}$ and the feasible region of the relaxed problem as $\hat{\mathcal{X}} = \{ x~\in~[0,1]^n \allowbreak ~|~ Ax~\ge~b \}$. We say that an element of $\hat{\mathcal{X}}$ is feasible and an element of $\{0,1\}^n$ is integer. A solution of Problem~\eqref{opt:binlp} is both feasible and integer. While continuous variables can be easily included in Problem~\eqref{opt:binlp}, we do not consider them to simplify our presentation.

\subsection{Feasibility pump.}
\label{sec:trad_pump}

The main goal of the feasibility pump is to find a feasible solution of Problem~\eqref{opt:binlp}. To achieve this, it alternates between solving the linear relaxation of the problem for a given objective, and rounding its solution until a solution that is both feasible and integer is found. This approach is motivated by the fact that the extreme points of $\mathcal{X}$ are also extreme points of $\hat{\mathcal{X}}$ when Problem~\eqref{opt:binlp} has only binary variables. Therefore, there always exists a cost vector $\bar{\theta}$ such that $\bar{x} \in \argmin\left\{\bar{\theta}\,^\top x ~|~ x\in \hat{\mathcal{X}}\right\}$ is a feasible integer solution, provided that $\mathcal{X} \neq \varnothing$.

\subsubsection{Algorithm.}

The feasibility pump of \citet{Fischetti2005feasibility} is summarized in Algorithm~\ref{alg:basic_pump}, where we denote by $x^{(k)}$ the solution obtained at the $k$-th iteration, and by $x_i$ the $i$-th component of solution~$x$. The algorithm starts by solving the linear relaxation of Problem~\eqref{opt:binlp} to obtain the feasible solution $\hat{x}^{(0)}$. If $\hat{x}^{(0)}$ is integer, it belongs to~$\mathcal{X}$ and the algorithm terminates. Otherwise, the integer solution $\round{x}^{(0)}$ is computed by rounding $x^*$ component-wise to the closest integer. If $\round{x}^{(0)}$ is feasible then it belongs to $\mathcal{X}$ and, again, the algorithm terminates. In any iteration, if the rounded solution $\round{x}^{(k)}$ is not feasible, the algorithm projects it onto the feasible domain $\hat{\mathcal{X}}$ and finds the feasible point $\hat{x}^{(k+1)}$. If $\hat{x}^{(k+1)}$ is integral the algorithm terminates, otherwise $\round{x}^{(k+1)}$ is computed as the rounding of $\hat{x}^{(k+1)}$. If $\round{x}^{(k)}$ is feasible the algorithm terminates, otherwise another iteration is performed. In a nutshell, the algorithm alternatively projects solutions onto the feasible and integer regions.

\begin{algorithm}
    \caption{Original Feasibility Pump.}
    \label{alg:basic_pump}
    \begin{algorithmic}
        \State \textbf{Given:} Problem~\eqref{opt:binlp}, distance function $\Delta$, maximum number of iterations $\nMax$
        \State \textbf{Initialize:} $\hat{x}^{(0)} \in \argmin_x\{ c^\top x:\;x\in \tilde{\mathcal{X}}\}$ and $k=0$
        \If{$\hat{x}^{(0)}$ is integer}
            \State \textbf{Return} $\hat{x}^{(0)}$
        \EndIf
        \State Round $\hat{x}^{(0)}$ to obtain an integral solution $\round{\hat{x}}^{(0)}$
        \If{$\round{\hat{x}}^{(0)}$ is feasible}
            \State \textbf{Return} $\round{\hat{x}}^{(0)}$
        \EndIf
        \While {$k<\nMax$}
            \State Increase iteration counter: $k = k+1$
            \State Compute the feasible solution $\hat{x}^{(k)} =\argmin_x\{ \Delta(x, \round{\hat{x}}^{(k-1)}):\; x\in \hat{\mathcal{X}}\} $
            \If{$\hat{x}^{(k)}$ is integer}
                \State \textbf{Return} $\hat{x}^{(k)}$
            \EndIf
            \State Round $\hat{x}^{(k)}$ to obtain $\round{\hat{x}}^{(k)}$
            \If{$\round{\hat{x}}^{(k)}$ is feasible}
                \State \textbf{Return} $\round{\hat{x}}^{(k)}$
            \EndIf
        \EndWhile
    \end{algorithmic}
\end{algorithm}

\subsubsection{Distance function.} Projecting the integer solution onto the feasible domain requires a distance measure. The feasibility pump uses the $\ell_1$ norm, also called Hamming distance. It is defined as $\Delta(x, \round{x}) = \sum_i |x_i-\round{x}_i|$. Given a fixed vector $\round{x}$, $\Delta(x, \round{x})$ is an affine function in $x$ and, since $\round{x}$ is a binary vector, it can be computed as
\begin{equation}
    \label{eq:hamming}
    \Delta(x, \round{x}) = \sum_{i:\round{x}_i=1} (1-x_i)+\sum_{i:\round{x}_i=0} x_i = \sum_{i=1}^n \round{x}_i - \sum_{i:\round{x}_i=1} x_i +\sum_{i:\round{x}_i=0} x_i.
\end{equation}
Hence, in iteration $k$, the feasible solution that is closest to $\round{x}^{(k-1)}$ is found by solving the linear program given by $\min\{\Delta(x, \round{x}^{(k-1)}) ~|~ x\in \hat{\mathcal{X}}\}$ or, equivalently, by solving
\begin{equation}
    \label{opt:theta_LP}
    \min\{\theta^\top x ~|~ x\in \hat{\mathcal{X}}\},
\end{equation}
where $\theta \in \reals^n$ is such that $\theta_i = -1$ if $\round{x}^{(k-1)}_i = 1$, and $\theta_i = 1$ otherwise. Equation~\eqref{opt:theta_LP} shows that the feasibility pump solves a sequence of linear problems that differ only in their cost parameters.

\subsubsection{Restarts.} A limitation of the feasibility pump is that it can get stuck in cycles when the objective given by Equation~\eqref{eq:hamming} has already been obtained in the past. Several strategies have been introduced to escape these cycles. The first strategy of \citet{Fischetti2005feasibility} is to flip a random number of entries of the rounded solution $\round{x}$ at the end of an iteration when a cycle of size one is detected. The second is to apply random perturbations to $\round{x}$ when a cycle of size larger than two is detected. These restart strategies, called \textit{flip} and \textit{perturbation} respectively, are not the focus of this paper. To simplify our presentation, we omitted them from Algorithm~\ref{alg:basic_pump} and assume thereafter that our feasibility pump algorithms implement them exactly as the original one.

\subsection{Differentiable optimization.}
\label{sec:diff_opt}

This work is motivated by the novel perspective brought by recent works on differentiable optimization. This stream studies how to integrate mixed-integer linear optimization models as differential layers in structured learning \citep{Amos2017optnet, Blondel2020learning, Blondel2024elements} and contextual stochastic optimization \citep{Elmachtoub2022, Sadana2023survey}. This allows training a machine-learning pipeline to predict a set of parameters that respect the constraints of a mixed-integer linear problem. Further, the machine-learning model can be trained in an end-to-end fashion: minimizing the cost of a downstream optimization task \citep[see, e.g.,][]{Donti2017, Elmachtoub2022, Wilder2019, Ferber2020}. To achieve this, a meaningful gradient must be propagated from the loss back to the machine-learning model. This can be done in several ways. Early works focus on quadratic programs and use an implicit differentiation approach based on the Karush–Kuhn–Tucker conditions of optimality \citep{Amos2017optnet, Amos2019differentiable}.

While several types of optimization models have been considered, the most studied form is the linear Problem~\eqref{opt:binlp} parameterized by its cost coefficients. Let $x^*(\theta)$ be a solution of the combinatorial linear problem $\min\{\theta^\top x ~|~ x\in \mathcal{X}\}$ parameterized by $\theta$. A key challenge is that the gradient of the $\argmin$ of a (combinatorial) linear program with respect to its cost vector is non-informative. Formally, the Jacobian $J_{\theta} x^*(\theta)$ is equal to~$0$ almost everywhere since the solution $x^*(\theta)$ is constant for small variations of~$\theta$, and undefined as the solution jumps from one vertex to another for a large enough change. Differentiable optimization studies how to obtain surrogate gradients.

A common technique is to consider a surrogate optimization model and using its gradient in place of the ones of the original model. The perturbation method of \citet{Berthet2020} regularizes the optimization model by perturbing its cost vector \citep{Dalle2022}. With an additive perturbation, the solution of the regularized problem is given by
\begin{equation*}
    x_\varepsilon^*(\theta) = \Exp_Z \left[ \argmin_{x\in \mathcal{X}} (\theta + \varepsilon Z)^\top x \right],
\end{equation*}
where $\varepsilon>0$ is a scaling parameter and $Z$ follows an exponential distribution with positive differentiable density. The Jacobian of $x_\varepsilon^*(\theta)$ is then given by
\begin{equation}
    \label{eq:pert_jacobian}
    J_\theta x_\varepsilon^*(\theta) = 1/\varepsilon \Exp_Z \left[ x_\varepsilon^*(\theta + Z) Z^\top \right].
\end{equation}
This gradient can be evaluated using a Monte Carlo approximation, by sampling perturbation vectors and solving the corresponding optimization problems. Other methods to obtain a differentiable surrogate optimization model include adding a quadratic regularization term \citep{Wilder2019}, using a linear interpolation between solutions \citep{vlastelica2019diff}, or linking gradient computations with interior point methods \citep{Mandi2020}.

A limitation of the perturbed optimizers is that they need to solve several problems for each gradient computation, which can be computationally expensive. An efficient alternative is to take minus the identity matrix as surrogate Jacobian \citep{Sahoo2023backpropagation, Fung2022jfb}, i.e.,
\begin{equation}
    \label{eq:minut_id}
    J_\theta x^*(\theta)= - I_n,
\end{equation}
with $I_n$ the identity matrix of size $n$. The core idea of this approach follows from the observation that when the $i$-th coefficient of the cost vector increases, it is likely that the $i$-th component of the optimal solution decreases.
\section{Differentiable Feasibility Pump}
\label{sec:generalized}

Differentiable optimization underlies our reinterpretation of the feasibility pump algorithm as a gradient-descent algorithm. Further, it is based on two key observations: (i)~the feasibility pump solves a sequence of problems that differ only in their cost parameters, and (ii)~the feasibility pump focuses precisely on the class of problems studied in differentiable optimization.

\subsection{Regularized gradient-descent interpretation.} 
We now present the differentiable feasibility pump, an interpretation of the feasibility pump as a regularized gradient-descent algorithm. This algorithm minimizes a loss function $\loss : \reals^n \rightarrow \reals$ over the cost vector $\theta \in \reals^n$. First, we present the differentiable feasibility pump algorithm shown in Algorithm~\ref{alg:gen_fp}, where $\eta > 0$ is the step size. Second, we show how a restricted version of this generalized pump recovers the original feasibility pump.
\begin{algorithm}
    \caption{Differentiable feasibility pump}
    \label{alg:gen_fp}
    \begin{algorithmic}
        \State \textbf{Initialize:} $\theta^{(0)} = c$, $k=0$
        \While {$k<\nMax$}
            \State Compute the feasible solution $\hat{x}^{(k)} = \argmin\{{\theta^{(k)}}^\top x: x\in \tilde{\mathcal{X}}\}$
            \State Round it to obtain the integer solution $\round{\hat{x}}^{(k)}$
            \If{$\round{\hat{x}}^{(k)}$ is feasible}
                \State \textbf{Return} $\round{\hat{x}}^{(k)}$
            \Else{}
                \State Update the cost vector $\theta^{(k+1)} = \theta^{(k)} - \eta \nabla_\theta \loss(\theta)$
            \EndIf
        \EndWhile
    \end{algorithmic}
\end{algorithm}

\begin{remark}
    Algorithm~\ref{alg:basic_pump} considers two success conditions: when the feasible solution is integral and when the rounded solution is feasible. However, they are redundant since the second criterion is always satisfied when the first one is. This is why Algorithm~\ref{alg:gen_fp} uses a single termination criterion: $\round{x}$ is feasible.
\end{remark}
\subsubsection{Integrality loss.}
\label{sec:integrality_loss}
The first success condition of the original feasibility pump is finding a cost vector $\bar{\theta}$ such that $\hat{x}(\bar{\theta}) = \argmin_x\{\bar{\theta}^\top x: x\in \hat{\mathcal{X}}\}$ is integer. We introduce the loss function $f : \reals^n \rightarrow \reals$ to measure the non-integrality of a feasible solution $\hat{x}$. In this work, we consider the class of integrality losses
\begin{equation}
    \label{eq:integrality_loss}
    f(x) = \sum_{i=1}^n \left( \min \left\lbrace x_i,1-x_i \right\rbrace \right)^p,
\end{equation}
where $p > 0$ is the loss order. As also studied by \citet{DeSantis2014feasibility}, a multitude of choices are possible such as $f(x) = \sum_{i=1}^n x_i(1-x_i)$.

Based on the value of $p$, the loss~$f$ has local convexity/concavity properties. Formally, consider $\mathcal{D} = \left\{\prod_{j=1}^n D_j ~|~ D_j \in \{[0, 0.5), [0.5, 1]\} \right\}$, the partition of the $n$-dimensional hypercube . The restriction of $f(x)$ to any element of $\mathcal{D}$ is concave if $p \leq 1$ and convex if $p \geq 1$. This is illustrated in Figure~\ref{fig:integ_loss} for a single decision variable.
\begin{figure}[ht!]
    \centering
    \resizebox{0.95\linewidth}{!}{\begin{tikzpicture}
\begin{groupplot}[
    group style={
        group name=my plots,
        group size=3 by 1,
        xlabels at=edge bottom,
        ylabels at=edge left},
    height = 4.5cm,
    width  = 5cm,
    enlarge x limits = 0,
    xlabel = {$x$},
    ylabel = {Integrality loss $f(x)$},
    xmin = 0,
    xmax = 1,
    ymin=0
    ]

    \nextgroupplot[title = {(a) $p=0.5$}, font = \small]
    \addplot+[mark=none, line width=0.3mm] table [x index = {0}, y index = {1}, col sep=comma]{plots/csv/integrality_loss.csv};

    \nextgroupplot[title = {(b) $p=1$}, font = \small]
    \addplot+[mark=none, line width=0.3mm] table [x index = {0}, y index = {2}, col sep=comma]{plots/csv/integrality_loss.csv};

    \nextgroupplot[title = {(c) $p=2$}, font = \small]
    \addplot+[mark=none, line width=0.3mm] table [x index = {0}, y index = {3}, col sep=comma]{plots/csv/integrality_loss.csv};

\end{groupplot}
\end{tikzpicture}}
    \caption{Integrality loss as a function of $x$ for several loss orders: (a)~$p=0.5$, (b)~$p=1$, and (c)~$p=2$.}
    \label{fig:integ_loss}
\end{figure}
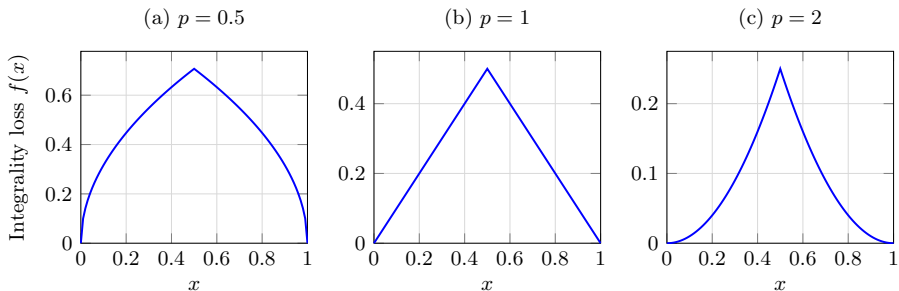

\subsubsection{Regularized loss.}
\label{sec:reg_loss}
Consider now the regularized loss function given by
\begin{equation}
    \label{eq:basic_loss}
    \loss(\theta) = f(\hat{x}(\theta)) + \gamma \frac{\|\theta\|_2^2}{2},
\end{equation}
where $\gamma>0$ is the regularization weight. The second component of this loss adds a regularization term that controls the magnitude of the cost vector $\theta$. Minimizing over this regularized loss is equivalent to minimizing the integrality loss $f$ with a constraint on the norm of $\theta$. Notice that the norm of $\theta$ does not influence the solution of the linear relaxation. That is $\hat{x}(\theta) = \hat{x}(\theta / r)$ for any scalar $r > 0$, with a slight abuse of notation in case the $\argmin$ is set-valued.

\subsubsection{Gradient update.} Using this regularized loss, the differentiable feasibility pump minimizes the optimization problem given by
\begin{subequations}
\label{opt:min_reg_loss}
\begin{alignat}{2}
\min_{\theta}    & \quad && \loss(\theta), \\
\text{s.t.} &       && \theta \in \reals^n.
\end{alignat}
\end{subequations}
Solving Problem~\eqref{opt:min_reg_loss} with a gradient-descent algorithm requires computing the gradient of the loss $\loss$ with respect to $\theta$, that is
\begin{equation}
    \nabla_{\theta} \loss(\theta) = \nabla_{\theta} \left(f(\hat{x}(\theta)) + \gamma \frac{\|\theta\|_2^2}{2}\right),
\end{equation}
and updating the cost vector $\theta$ as
\begin{equation}
    \label{eq:grad_update}
    \theta \Longleftarrow \theta - \eta \nabla_{\theta} \loss(\theta).
\end{equation}
Using the chain rule, the gradient of the loss can be computed as
\begin{equation}
    \label{eq:chain_rule}
    \nabla_{\theta} \loss(\theta) = (J_{\theta} \hat{x}(\theta))^\top \nabla_{\hat{x}(\theta)}f(\hat{x}(\theta)) + \gamma \theta.
\end{equation}
As discussed in Section~\ref{sec:diff_opt}, several techniques can be applied to obtain a surrogate Jacobian $J_\theta \hat{x}(\theta)$.

\subsection{Equivalence with the original algorithm.}
We now show that the original feasibility pump is a special case of the differentiable pump with specific parameters, loss function, and gradient estimation.
\begin{proposition}
    \label{prop:equivalence}
    When $f(x) = \sum_{i=1}^n \min\{x_i, 1-x_i\}$, $\eta = 1$, $\gamma = 1$, and the Jacobian of the $\argmin$ is computed as $J_{\theta} \hat{x}(\theta) = - I_n $, Algorithm~\ref{alg:gen_fp} generate the same sequence of solutions $\{ \hat{x}^{(k)} \}$ as the original feasibility pump algorithm.
\end{proposition}
\begin{proof}
    Consider the regularized loss given in Equation~\eqref{eq:basic_loss} with integrality loss as in Equation~\eqref{eq:integrality_loss} with $p=1$. The gradient update is now given by
    \begin{equation*}
        \theta \Longleftarrow \theta - \eta \nabla_\theta \left( f(\hat{x}(\theta)) \right) - \eta\gamma \theta.
    \end{equation*}
    Choosing a specific set of parameters such that $\eta\gamma=1$, for example $\eta = 1$ and $\gamma=1$, the gradient update becomes $\theta \Longleftarrow -  \nabla_\theta \left( f(\hat{x}(\theta)) \right)$. The updated cost vector is now equal to minus the gradient of the integrality loss. Using the chain rule, we obtain 
    \begin{equation}
         \theta \Longleftarrow -  (J_{\theta} \hat{x}(\theta))^\top \cdot \nabla_{\hat{x}(\theta)} f(\hat{x}(\theta)),
    \end{equation}
    which simplifies to $\theta \Longleftarrow \nabla_{\hat{x}(\theta)} f(\hat{x}(\theta))$ when taking the Jacobian as minus the identity matrix as in Equation~\eqref{eq:minut_id}. The integrality loss with $p=1$ is piecewise linear. Its gradient is given by
    \begin{equation}
        \label{eq:grad_minx1mx}
        \nabla_{\hat{x}(\theta)} f(\hat{x}(\theta))= \begin{cases}
                                    -1 & \text{if } \hat{x}(\theta) > 0.5,\\
                                    1  & \text{otherwise.}
                                \end{cases}
    \end{equation}
    In Section \ref{sec:trad_pump}, we noticed that the original pump solves a sequence of problems $\argmin\{\theta^\top x ~|~ x\in \hat{\mathcal{X}}\}$, where $\theta_i= -1$ if $\round{x}_i =1$ and $\theta_i= 1$ otherwise. Hence, the sequences of problems solved by the two algorithms are equivalent.
\end{proof}

\begin{remark}
    For conciseness, we do not discuss the cycle detection and restarts operations in the proof of Proposition~\ref{prop:equivalence}. The cycle detection can be kept the same as the original algorithm since the same information is available. It is direct to define the restart operations by flipping or perturbing $\theta$ rather than $\round{x}$. Flipping the $i$-th coordinate of $\round{x}$ (i.e., changing its value from 1 to 0 or vice-versa) is equivalent to multiplying the $i$-th coordinate of $\theta$ by $-1$.
\end{remark}
Proposition~\ref{prop:equivalence} shows that the feasibility pump can be seen as a gradient descent with a very specific set of parameters. In particular, the algorithm parameters are such that $\eta\gamma=1$. An interesting consequence of these parameters is that the new cost vector obtained following the gradient update depends only on the gradient of the loss.

\textbf{Novel perspective.} The differentiable feasibility pump interpretation opens many possibilities to alter, extend, and improve the algorithm. For instance, without any change to the algorithm, it allows tuning the main hyper-parameters such as the step size $\eta$ and the regularization weight $\gamma$.

Further, it allows reassessing the loss function chosen to measure the non-integrality of a feasible solution. The loss $f(x) = \sum_i \min\{x_i,1-x_i\}$ (i.e., with order $p=1$) is piecewise linear. Therefore, it has a piecewise constant gradient. This means that the gradient of the integrality loss is independent of the magnitude of the integrality loss. This might lead to an unstable behavior. Intuitively, we would like the gradient of the integrality loss to be small when the loss is small, and large when the loss is large. This can be achieved by using a convex integrality loss, for instance by taking $p > 1$ as illustrated in Figure~\ref{fig:integ_loss}. Notice that this recommendation contrasts with the existing literature, especially the work of \citet{DeSantis2014feasibility}. The authors investigate concave loss functions, justified by their reinterpretation of the feasibility pump algorithm as a Frank-Wolfe algorithm. In contrast, our interpretation as a gradient descent justifies using a convex loss. Finally, it is notable that the Jacobian surrogate $J_\theta \hat{x}(\theta) = - I_n$ is used implicitly even though it was formalized around twenty years after the feasibility pump algorithm.
\section{Extending the differentiable pump}
\label{sec:extension}

In this section, we extend the differentiable pump beyond a simple reinterpretation. We present a generalized loss function and introduce a novel loss component to take into account the infeasibility of the rounded solution. We discuss the motivation for this loss and show how to differentiate it efficiently by introducing a differentiable rounding operation. Finally, we conclude our methodology by showing how this general framework encompasses several existing feasibility pump algorithms.

\subsection{General loss function.}
We consider the general loss given by
\begin{equation}
    \label{eq:gen_loss}
    \loss(\theta) = \alpha C(\theta) +  \beta f(\hat{x}(\theta)) + \lambda g(\round{\hat{x}(\theta)}) + \gamma \Omega(\theta),
\end{equation}
where $(\alpha, \beta, \lambda, \gamma)$ is a vector of positive weights. This general loss has four components. The first component~$C$ measures the cost of the solutions of the feasibility pump with respect to the objective of the problem. It is motivated by the idea that a feasibility pump should provide a primal solution that is not only feasible but also of good quality, i.e., with a low cost \citep{Achterberg2007improving}. It can be taken for instance as $C(\theta) = c^\top \hat{x}(\theta)$, $C(\theta) = c^\top \round{\hat{x}(\theta)}$, or a convex combination of them. The second component~$f$ aims to minimize the non-integrality of the solution. Hence, it depends only on the feasible solution $\hat{x}(\theta)$. The third component~$g$ is a novel loss function that depends only on the rounded solution $\round{\hat{x}(\theta)}$. It minimizes the infeasibility of the rounded solution, thus mirroring the previous term. Finally, the fourth term is the regularization of the cost vector. The remainder of this section presents a feasibility loss and discusses how to adapt the rounding operation to propagate a meaningful gradient.

The general loss function in Equation~\eqref{eq:gen_loss} induces the computational graph shown in \cref{fig:gen_pipeline}. This figure highlights the two main operations of the algorithm: solving the linear relaxation and rounding the feasible solution. Performing a gradient-descent step on $\loss(\theta)$ implies differentiating through these two operations.
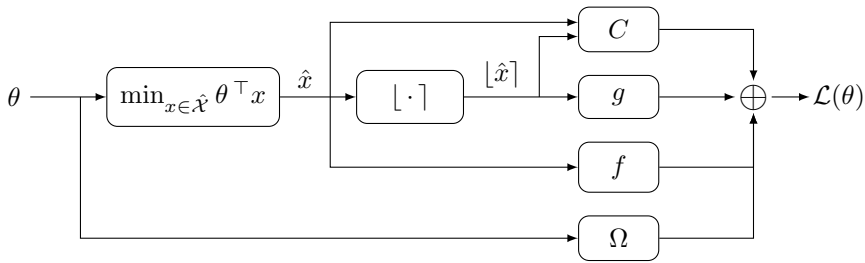
\begin{figure}[ht!]
    \centering
    \begin{tikzpicture}[
      auto,
      >=latex,
]
\tikzstyle{bloc} = [
      rectangle,
      draw,
      text centered,
      inner sep=0.5em,
      rounded corners,
      minimum width=4em,
      minimum height=2em
]
\tikzstyle{loss} = [
      bloc,
      minimum width=3em,
      minimum height=1.5em
]
\node (input) {$\theta$};
\node (argmin) [bloc, right=1cm of input.east] {$\min_{x \in \hat{\mathcal{X}}} \theta\,^\top x$};
\node (round) [bloc, right=1cm of argmin.east, anchor=west] {$\round{\, \cdot \,}$};
\node (feas_loss) [loss, right=1.5cm of round.east, anchor=west] {$g$};
\node (cost_loss) [loss, above = 0.3cm of feas_loss] {$C$};
\node (int_loss) [loss, below = 0.3cm of feas_loss] {$f$};
\node (reg_loss) [loss, below = 0.3cm of int_loss] {$\Omega$};
\node (sum) [inner sep = 1pt, right=1cm of feas_loss.east] {\LARGE $\oplus$};
\node (stop) [inner sep = 1pt, right=0.5cm of sum.east] {$\loss(\theta)$};

\draw[->] (input) -- (argmin);
\draw[->] ($(input.east)!0.66!(argmin.west)$) |- (reg_loss);
\draw[->] (argmin) -- node [above, pos=0.33] (x_lp) {$\hat{x}$} (round);
\draw[->] ($(argmin.east)!0.66!(round.west)$) |- (int_loss);
\draw[->] ($(argmin.east)!0.66!(round.west)$) |- (cost_loss.170);
\draw[->] (round) -- node [above, pos=0.33] (x_lp) {$\round{\hat{x}}$} (feas_loss);
\draw[->] ($(round.east)!0.66!(feas_loss.west)$) |- (cost_loss.190);
\draw[->] (cost_loss) -| (sum);
\draw[->] (feas_loss) -- (sum);
\draw[->] (int_loss) -| (sum);
\draw[->] (reg_loss) -| (sum);
\draw[->] (sum) -- (stop);
\end{tikzpicture}%
    \caption{Computational graph of the differentiable feasibility pump algorithm.}
    \label{fig:gen_pipeline}
\end{figure}

\subsection{Feasibility loss.}
The main novelty of this generalized loss is taking into account the infeasibility of the rounded solution $\round{\hat{x}(\theta)}$. As previously discussed, most existing approaches only consider the non-integrality of $\hat{x}(\theta)$, and the update of the cost vector at each iteration is agnostic to the feasibility of the rounded solution. A notable exception is the approach of \citet{Frangioni2021constraints} that encourages the satisfaction of the constraints by modifying the integrality loss to take into account the slack and active constraints of the feasible solution.

Again, many loss functions could be used to measure the infeasibility of the rounded solution. A suitable loss function should be: (i)~equal to zero when all constraints are satisfied, (ii)~easily differentiable, and (iii)~not heavily impacted by the number of constraints or magnitude of the constraint matrices. To achieve this, we measure the infeasibility of the rounded solution element-wise, normalize the infeasibility by the constraint coefficients, and average over the constraints.

Let $s_j(\tilde{x}): \{0, 1\}^n \rightarrow \reals$ measure the slackness of the $j$-th constraint as
\begin{equation*}
    s_j(\tilde{x}) = \frac{b_j - A_j \tilde{x}}{\|[A_j \, b_j]\|_2},
\end{equation*}
where $A_j$ is the $j-$th row of the constraint matrix, $b_j$ is the $j-$th left-hand side coefficient, and $[A_j \, b_j]$ is their concatenation. We define $g_j: \mathcal{X} \rightarrow \reals $ to measure the infeasibility of the $j$-th constraint as
\begin{equation*}
    g_j(\tilde{x}) = \max(s_j(\tilde{x}) - \epsilon, 0) =  \begin{cases}
                        s_j(\tilde{x}), & \text{ if $\tilde{x}$ violates the $j$-th constraint},\\
                        0,  & \text{ otherwise,}
                    \end{cases}
\end{equation*}
where $\epsilon>0$ is a small constant introduced to avoid numerical instability. This function is known in the machine-learning community as the rectified linear unit~(ReLU). Hence, it can be compactly rewritten as $g_j(\tilde{x}) = \relu ( s_j(\tilde{x}) - \epsilon )$. Finally, we average the infeasibility over the constraints by taking
\begin{equation}
    \label{eq:feas_loss}
    g(\tilde{x})= \frac{\sum_{j=1}^m g_j(\tilde{x})}{m}.
\end{equation}


\subsection{Differentiable rounding.}
\label{sec:diff_round}
The last step needed to obtain a fully differentiable feasibility pump algorithm is to find a surrogate gradient for the rounding operation. Here, we apply again techniques of differentiable optimization. Using perturbation, we show that this gradient can be expressed in closed-form and is especially efficient to compute.

Rounding is applied element-wise. Hence, we consider that $x$ is single dimensional. Our approach is based on the observation that rounding can be seen as a linear program, i.e., $\round{x}$ is the solution of the optimization problem given by
\begin{equation}
    \label{opt:round}
    \max_{y} \{ (x - 0.5)^\top y ~|~ y \in [0, 1]\}.
\end{equation}
This problem depends only on $x$ in its objective coefficients. Thus, it fits the class of problems studied extensively in the differentiable optimization literature. The rounding problem can be differentiated by applying the techniques discussed in Section~\ref{sec:diff_opt} such as taking the identity matrix as Jacobian. A more informative gradient can be obtained using perturbation.

Define the perturbed rounding operation as
\begin{equation}
    \label{eq:pert_round}
    \round{x}_\varepsilon = \Exp_Z [ \round{x + \varepsilon Z} ],
\end{equation}
where $Z$ is a random variable with positive and differentiable density on $\mathbb{R}$. Using Equation~\eqref{eq:pert_jacobian}, the Jacobian of the perturbed rounded solution can be computed as
\begin{equation}
    \label{eq:pertb_jac}
    J_x \round{x}_\varepsilon = \frac{1}{\varepsilon} \Exp_Z [Z \mid x+\varepsilon Z > 0.5] \Prob(x+\varepsilon Z > 0.5).
\end{equation}
This expression is obtained by Bayes's law. The expectation is also known as the inverse Mills ratio. The gradient of the perturbed rounded solution can be obtained in closed-form as 
\begin{equation}
    \label{eq:round_grad}
    J_x \round{x}_\varepsilon = \frac{1}{\varepsilon} \phi \left( \frac{0.5-x}{\varepsilon} \right),
\end{equation}
with $\phi$ the probability density function~(p.d.f.) of the standard Gaussian. Remarkably, we do not need to evaluate the expectation in Equation~\eqref{eq:pertb_jac} using a Monte Carlo approximation: it is sufficient to compute the p.d.f. of a standard Gaussian, which is very efficient. Hence, introducing a feasibility loss based on perturbed rounding does not increase the complexity of the feasibility pump. The computational bottleneck remains solving the relaxed problem $\hat{x}(\theta)$.

\begin{remark}
    This expression can be also obtained by observing that the perturbed rounding can be equivalently expressed as $\round{x}_\varepsilon = \Exp_Z [ \indicator(x + \varepsilon Z \ge 0.5) ]$ so that $\round{x}_\varepsilon = \Phi_Z(\frac{x-0.5}{\varepsilon}) $ with $\Phi_Z$ the cumulative distribution function of $Z$. This expression holds for any perturbation distribution in the exponential family. Its gradient can thus be directly deduced.
\end{remark}

\textbf{Soft rounding.} We call this differentiable rounding operator a \textit{soft} rounding, akin to \textsf{softmax} being a differentiable surrogate to the $\max$ operator. Other soft operators commonly used in machine learning include the differentiable top-$k$ layer \citep{Xie2020differentiable} or differentiable sorting \citep{Prillo2020softsort}. Soft rounding is also closely linked to straight-through estimators \citep{Bengio2013estimating}.

In the chain rules, the gradient of the perturbed rounding is can be seen as a smoothing term: the larger the influence of the rounding, the larger the gradient of the loss. To illustrate this result, we show the gradient as a function of $x$ in \cref{fig:soft_rounding}. The figure also shows how the perturbed rounding varies with the magnitude of the perturbation. As $\varepsilon$ increases, the soft rounding converges to a linear function and its gradient becomes constant, thus recovering the approximation of taking the identity matrix as in \cref{eq:minut_id}. Conversely, as $\varepsilon$ goes toward zero, the soft rounding gets closer to the true rounding, and the gradient sharpens toward the Dirac function at $x=0.5$.
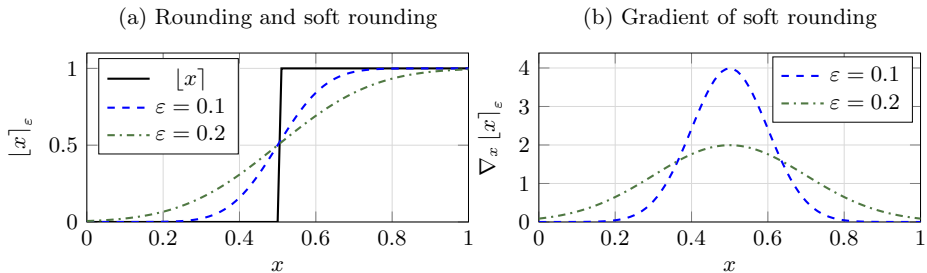
\begin{figure}[ht!]
    \centering
    \resizebox{0.98\linewidth}{!}{\begin{tikzpicture}
\begin{groupplot}[
    group style={
        group name=my plots,
        group size=2 by 1,
        xlabels at=edge bottom,
        ylabels at=edge left},
    height = 4cm,
    width  = 7cm,
    enlarge x limits = 0,
    xlabel = {$x$},
    xmin = 0,
    xmax = 1,
    ymin=0
    ]

    \nextgroupplot[title = {(a) Rounding and soft rounding}, font = \small, ylabel = {$\round{x}_\varepsilon$}, legend pos=north west]
    \addplot+[mark=none, line width=0.3mm, black] table [x index = {0}, y index = {1}, col sep=comma]{plots/csv/soft_rounding.csv};
    \addlegendentry{$\round{x}$}
    \addplot+[mark=none, line width=0.3mm, blue, dashed] table [x index = {0}, y index = {2}, col sep=comma]{plots/csv/soft_rounding.csv};
    \addlegendentry{$\varepsilon = 0.1$}
    \addplot+[mark=none, line width=0.3mm, darkgreen, dashdotted] table [x index = {0}, y index = {3}, col sep=comma]{plots/csv/soft_rounding.csv};
    \addlegendentry{$\varepsilon = 0.2$}

    \nextgroupplot[title = {(b) Gradient of soft rounding}, font = \small, ylabel = {$\nabla_x \round{x}_\varepsilon$}]
    \addplot+[mark=none, line width=0.3mm, blue, dashed] table [x index = {0}, y index = {4}, col sep=comma]{plots/csv/soft_rounding.csv};
    \addlegendentry{$\varepsilon = 0.1$}
    \addplot+[mark=none, line width=0.3mm, darkgreen, dashdotted] table [x index = {0}, y index = {5}, col sep=comma]{plots/csv/soft_rounding.csv};
    \addlegendentry{$\varepsilon = 0.2$}

\end{groupplot}
\end{tikzpicture}}
    \caption{Soft rounding and its gradient for varying $\varepsilon$.}
    \label{fig:soft_rounding}
\end{figure}

\subsection{Reductions to existing algorithms.}
We specify here how many versions of the feasibility pump algorithm can be seen as specific cases of our generalization. To show this, we first need to observe that Algorithm~\ref{alg:gen_fp} can still be applied when modifying the rounding or linear relaxation. Such modifications include for instance \citet{Fischetti2009feasibility}, who modify the rounding operation to take into account the feasibility of the rounded solution.

Hence, we only need to show how the generalized loss in Equation~\eqref{eq:gen_loss} can be specialized to recover existing algorithms, such as the one of \citet{Achterberg2007improving} who introduce a term to account for the cost of the solution, or \citet{Frangioni2021constraints} who introduce feasibility considerations into the loss function by reweighting the non-integrality loss depending on the active constraints of the solution $\hat{x}(\theta)$.

In \cref{tab:equivalences}, we show how the differentiable pump can reduce to existing variants of the feasibility pump. All these variants are based on the special case of taking $J_{\theta}\hat{x}(\theta)=-I$ as Jacobian of the linear solution, and regularizing the cost vector $\theta$ with parameters such that $\gamma \eta = 1$. The table shows that most variations of the feasibility pump extend the integrality loss or modify the restart or rounding components. A tick symbol~(\tick) indicates that a component of the feasibility pump is modified compared to the original algorithm whereas a cross symbol~(\cross) indicates that the component is kept as the original one.
\begin{table}[!ht]
    \caption{Existing feasibility pumps as special cases of the differentiable pump.}
    \label{tab:equivalences}
    \centerline{
    \begin{threeparttable}
        \small
        \begin{tabular}{rcccc}
            \toprule
                    & \multicolumn{2}{c}{Loss function} & \multicolumn{2}{c}{Modified algorithm}\\
            \cmidrule(lr){2-3}\cmidrule(lr){4-5}
                    & Integrality loss $f(x)$ & Cost $C(\theta)$ & Restart & Rounding \\
            \midrule
            \citet{Achterberg2007improving} & $\sum_{i=1}^n \min\{x_i, 1-x_i\}$     & $c^\top \hat{x}(\theta)$  & \cross  & \cross \\
            \citet{Fischetti2009feasibility}& $\sum_{i=1}^n \min\{x_i, 1-x_i\}$     & \cross                & \cross  & \tick  \\
            \citet{Baena2011using}          & $\sum_{i=1}^n \min\{x_i, 1-x_i\}$     & $c^\top \hat{x}(\theta)$  & \cross  & \tick  \\
            \citet{DeSantis2013new}         & $\lambda f_1(x) + (1-\lambda) f_2(x)$ & $c^\top \hat{x}(\theta)$  & \tick   & \cross \\
            \citet{Geissler2017penalty}     & $\sum_{i=1}^n \mu_i\min\{x_i, 1-x_i\}$  & $c^\top \hat{x}(\theta)$  & \tick   & \tick  \\
            \citet{Dey2018improving}        & $\sum_{i=1}^n \min\{x_i, 1-x_i\}$     & \cross                & \tick   & \cross \\
            \citet{Frangioni2021constraints}& $f_3(x, Ax-b)$                        & \cross                & \cross  & \cross \\
            \bottomrule
        \end{tabular}
        \begin{tablenotes}
            \small
            \item Note: $f_1$ and $f_2$ are concave losses based on exponential, logistic, hyperbolic, or logarithmic functions. $\{\mu_i\}_{i=1}^n$ is a set of positive weights that increase monotonously with the iterations of the algorithms. $f_3$ is a non-convex function that depends on the slack and active constraints of the feasible solution.
        \end{tablenotes}
    \end{threeparttable}
    }
\end{table}
\section{Numerical Study}
\label{sec:numerical}

In this section, we investigate the performance of the differentiable feasibility pump. First, we show that substantial performance improvements can be obtained when modifying the parameters of the original algorithm. Second, we investigate the value of the feasibility loss introduced in \cref{sec:extension}. Finally, we analyze the sensitivity of feasibility pump algorithms to variations in some of the key steps highlighted in \cref{sec:generalized} such as the gradient-descent algorithm and the Jacobian surrogate.

We run experiments on a set of $125$ instances from MIPLIB 2017 \citep{Gleixner2021}. We select instances with continuous or binary variables and exclude exceptionally large instances. The list of instances used is given in \cref{app:num}. All linear problems are solved using the commercial solver \texttt{Gurobi}. To ensure consistent and reliable results, we tune the solver's parameters to use a single thread with the smallest possible tolerance for optimality and feasibility. The differentiable pump is implemented in \texttt{python} using the \texttt{pytorch} package for automatic differentiation. The differentiable optimization models are adapted from the \texttt{PyEPO} package of \citet{tang2022pyepo}. The code used to run all the experiments in this paper is openly accessible at \url{https://github.com/MCacciola94/DiffPump}.

\subsection{Experimental setting.}
We investigate several variants of the differentiable feasibility pump, with increasing flexibility compared to the original algorithm. All methods use the general loss function given in Equation~\eqref{eq:gen_loss} with a specific set of weights $(\beta, \lambda, \gamma)$ and step size $\eta$. We do not consider the cost component of the loss (i.e., we set $\alpha = 0$) in order to better compare the results of the differentiable pump with the ones of the original algorithm.

The integrality loss~$f$ is taken as in \cref{eq:integrality_loss} with $p$ denoting its order. The feasibility loss~$g$ is taken as in \cref{eq:feas_loss}. The regularization is given by $\Omega(\theta) = \|\theta\|^2 / 2$. Unless otherwise specified, all methods use minus the identity matrix as a surrogate for the Jacobian of the $\argmin$ of the linear relaxation. All variants using the feasibility loss use soft rounding with $\varepsilon = 0.15$. All methods use the same cycle detection method, restarts, and stopping criteria. Finally, all variants use a limit of $\nMax = 1000$ iterations. If a pump cannot find a feasible solution within the iteration limit, the experiment is considered failed.

The original algorithm is denoted as \texttt{FP}. We investigate four variants of the differentiable feasibility pump with increasing flexibility compared to \texttt{FP}. The first variant of the differentiable pump is denoted by \texttt{DP1}. It changes the weight of the regularization loss~$\gamma$ and the step size~$\eta$. The second variant \texttt{DP2} additionally changes the order~$p$ of the integrality loss. The third variant \texttt{DP3} uses only the feasibility loss with soft rounding as described in \cref{sec:diff_round}. The last variant \texttt{DP4} combines both loss functions with positive weights $\beta$ and $\lambda$ respectively, and is allowed to vary all hyperparameters. The selected methods and their best hyperparameters found using a grid search are shown in \cref{app:num}. \cref{tab:method_notation} summarizes the feasibility pumps used in our experiments, as well as their best hyperparameter values found using a grid search.
\begin{table}[H]
    \centering
    \caption{Variants of the differentiable feasibility pump used in main experiments and their best hyperparameter values.}
    \label{tab:method_notation}
    \begin{tabular}{llccccc}
       \toprule
             & Description                          & $\eta$ & $\beta$ & $\lambda$ & $\gamma$ & $p$\\
       \midrule
       \texttt{FP}  & Original feasibility pump                 &  1    &  1   &  0     &  1    & 1\\
       \texttt{DP1} & Change only $\gamma$ and $\eta$           &  1    &  1   &  0     &  0.95 & 1\\
       \texttt{DP2} & Change $\gamma$, $\eta$, and $p$          &  0.8  &  1   &  0     &  0.1  & 2\\
       \texttt{DP3} & Use only feasibility loss                 &  0.3  &  0   &  1     &  1    & -\\
       \texttt{DP4} & Combine integrality and feasibility loss  &  0.6  & 10   &  1e-3  &  0.1  & 2\\
       \bottomrule
    \end{tabular}
\end{table}

We consider three main performance metrics: (i)~the number of instances for which the algorithm did not find a solution within the limited number of iterations, (ii)~the total number of iterations of the algorithm over all instances, and (iii)~the percentage of iterations in which a restart operation was performed. For instance, a ratio of $50\%$ means that the algorithm performed a perturbation or flip operation every two iterations on average.

\subsection{Main results.}
Our main experimental results are presented in \cref{tab:main_results}. The table shows the number of fails, the total number of iterations, and the restart ratio for all feasibility pump variants. It shows that several variants of the differentiable feasibility pump achieve better performance than the original algorithm.
\begin{table}[!ht]
    \centering
    \caption{Main experimental results: number of fails, iterations, and randomization ratio over all instances.}
    \label{tab:main_results}
    \begin{tabular}{lccccc}
        \toprule
                      & \texttt{FP} & \texttt{DP1} & \texttt{DP2} &  \texttt{DP3} & \texttt{DP4}\\
        \midrule
        Number of fails             & 28        & 28        & 21        &   21      & 20        \\
        Total number of iterations  & $30071$ & $30037$ & $24084$ & $25684$ & $23383$ \\
        Restart ratio         & $43.71\%$ & $42.79\%$ & $24.61\%$ & $3.9\%$   & $25.98\%$ \\
        \bottomrule
    \end{tabular}
\end{table}

A first observation is that simply changing the gradient-update step size $\eta$ and the weight of the regularization loss $\gamma$ has minimal impact: the performance of \texttt{DP1} stays very close to the one of the original pump. The fact that changing the values of the parameters does not bring any significant improvement is remarkable. It suggests that the parameters of the original algorithm are a local optimum. In fact, \cref{app:num} shows that the best parameters of \texttt{DP1} are almost identical to the one of~\texttt{FP}.
\begin{result}
    The parameters of the original feasibility pump are a local optimum.
\end{result}

However, Table~\ref{tab:main_results} shows that changing the order of the integrality loss to $p=2$ and jointly adapting the step size and regularization weight leads to significant improvements both in the number of iterations and the number of instances solved. Indeed, \texttt{DP2} solves $7$ more instances than \texttt{FP} and uses more than $5000$ fewer iterations, a reduction of around $20\%$. Further, it leads to fewer flipping and restarts, having almost half the restart ratio of the original algorithm. This suggests that the new loss paired with $\eta \gamma \neq 1$ make the algorithm less likely to get stuck in a cycle.
\begin{result}
    The integrality loss with order $p=2$ leads to a more performant and stable feasibility pump. Our hypothesis that using a locally convex loss improves gradient descent is supported experimentally.
\end{result}

The third variant, \texttt{DP3}, also outperforms the original algorithm \texttt{FP}. \cref{tab:main_results} shows that it achieves almost the same performance of \texttt{DP2} but is substantially more stable. Indeed, its restart ratio is one order of magnitude smaller than that of \texttt{FP}. \texttt{DP3} behaves more closely to a pure gradient-descent algorithm as it only rarely uses the heuristic restart operations. This suggests that the loss landscape induced by the feasibility loss has fewer local minima.

\begin{result}
    Using only the feasibility loss also provides significant improvements compared to the original algorithm. It leads to an especially stable algorithm that only rarely uses restart operations.
\end{result}

Finally, \cref{tab:main_results} shows that \texttt{DP4}, which uses a combination of feasibility loss and integrality loss can provide a further improvement, although it is not substantial. The algorithm can find the feasible solution of one more instance and slightly reduce the total number of iterations. Interestingly, its restart ratio is similar to the one of~$\texttt{DP2}$.
\begin{result}
    Combining the feasibility and integrality losses can provide a small performance increase compared to using them individually.
\end{result}

\subsection{Sensitivity analyses.}
We now study the influence of key components of our algorithm. First, we provide a detailed analysis of how the performance of a feasibility pump varies with its hyperparameters. Then, we investigate changing the estimation of the Jacobian of the $\argmin$. Finally, we study how different gradient-descent algorithms can be implemented. Additional sensitivity results are provided in \cref{app:num}.

\subsubsection{Sensitivity to hyperparameters.}
The results of the hyperparameter grid search provide a sensitivity analysis of the impact of the hyperparameters $\eta$ and $\gamma$ on \texttt{DP1}. In \cref{fig:abl}, we show the impact of changing a single parameter while keeping the other equal to $1$. The figure shows that modifying a single parameter does not improve the original pump algorithm. It confirms our observation that the original algorithm is remarkably robust to its hyperparameters and is locally optimal.
\begin{figure}[ht]
    \centering
    \resizebox{0.98\linewidth}{!}{\begin{tikzpicture}
\begin{groupplot}[
    group style={
        group name=my plots,
        group size=2 by 1,
        xlabels at=edge bottom,
        ylabels at=edge left},
    height = 4cm,
    width  = 7cm,
    enlarge x limits = 0,
    enlarge y limits = 0,
    xmin = 0.2,
    xmax = 1,
    ylabel={Number of fails}
    ]

    \nextgroupplot[title = {(a) $\eta = 1$}, font = \small, xlabel=Regularization weight $\gamma$]
    \addplot table [x index = {0}, y index = {1}, col sep=comma]{plots/csv/fp_sensitivity.csv};

    \nextgroupplot[title = {(b) $\gamma = 1$}, font = \small, xlabel=Step size $\eta$]
    \addplot table [x index = {2}, y index = {3}, col sep=comma]{plots/csv/fp_sensitivity.csv};

\end{groupplot}
\end{tikzpicture}}
    \caption{Number of instances that cannot be solved by \texttt{DP1} for varying values of $\eta$ and $\gamma$.}
    \label{fig:abl}
\end{figure}
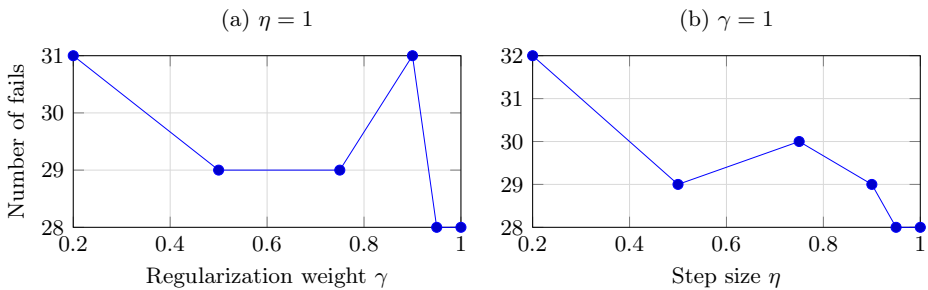

The results of the hyperparameter grid search over $\eta$ and $\gamma$ when using $p=1$ and $p=2$ are provided in \cref{tab:p1_grid} and \cref{tab:p2_grid}, respectively. They show that no significant improvement can be obtained when keeping $p=1$, whereas many configurations of $\texttt{DP2}$ that use $p=2$ outperform \texttt{FP}. In fact, all sets of parameters shown in \cref{tab:p2_grid} provide better performance than the original algorithm \texttt{FP}. This confirms that using a locally convex integrality loss function is more suitable for gradient descent.
\begin{table}[ht]
\centering
\caption{Sensitivity analysis of \texttt{DP1}.}
\label{tab:p1_grid}
\small
\begin{minipage}[t]{.3\linewidth}
\begin{tabular}{cccc}
    \toprule
    $\eta$ & $\gamma$ & Iters & Fails\\ 
    \midrule
    0.2 & 0.2 & 39492 & 35\\ 
    0.2 & 0.5 & 37617 & 35\\ 
    0.2 & 0.7 & 38878 & 36\\ 
    0.2 & 0.9 & 33849 & 31\\ 
    0.2 & 0.95 & 34781 & 30\\ 
    0.2 & 1 & 36489 & 32\\ 
    0.5 & 0.2 & 36130 & 33\\ 
    0.5 & 0.5 & 33130 & 31\\ 
    0.5 & 0.7 & 32692 & 30\\ 
    0.5 & 0.9 & 33717 & 30\\ 
    0.5 & 0.95 & 32994 & 29\\ 
    0.5 & 1 & 32405 & 29\\ 
    \bottomrule
\end{tabular}
\end{minipage}\hfill
\begin{minipage}[t]{.3\linewidth}
\begin{tabular}{cccc}
    \toprule
    $\eta$ & $\gamma$ & Iters & Fails\\ 
    \midrule
    0.7 & 0.2 & 36329 & 32\\ 
    0.7 & 0.5 & 31610 & 27\\ 
    0.7 & 0.7 & 32066 & 29\\ 
    0.7 & 0.9 & 33288 & 29\\ 
    0.7 & 0.95 & 33506 & 29\\ 
    0.7 & 1 & 32825 & 30\\
    0.9 & 0.2 & 33983 & 30\\ 
    0.9 & 0.5 & 33653 & 30\\ 
    0.9 & 0.7 & 31820 & 30\\ 
    0.9 & 0.9 & 32857 & 30\\ 
    0.9 & 0.95 & 31275 & 28\\ 
    0.9 & 1 & 32635 & 29\\ 
    \bottomrule
\end{tabular}
\end{minipage}\hfill
\begin{minipage}[t]{.3\linewidth}
\begin{tabular}{cccc}
    \toprule
    $\eta$ & $\gamma$ & Iters & Fails\\ 
    \midrule
    0.95 & 0.2 & 35424 & 32\\ 
    0.95 & 0.5 & 32940 & 30\\ 
    0.95 & 0.7 & 32382 & 29\\ 
    0.95 & 0.9 & 31451 & 29\\ 
    0.95 & 0.95 & 32886 & 30\\ 
    0.95 & 1 & 31411 & 28\\ 
    1 & 0.2 & 34551 & 31\\ 
    1 & 0.5 & 31620 & 29\\ 
    1 & 0.7 & 32287 & 29\\ 
    1 & 0.9 & 34064 & 31\\ 
    \textbf{1} & \textbf{0.95} & \textbf{30037} & \textbf{28}\\ 
    1 & 1 & 30071 & 28\\ 
    \bottomrule
\end{tabular}
\end{minipage}
\end{table}

\begin{table}[ht]
\centering
\caption{Sensitivity analysis of \texttt{DP2}.}
\label{tab:p2_grid}
\small
\begin{minipage}[t]{.45\linewidth}
\centering
\begin{tabular}{cccc}
    \toprule
    $\eta$ & $\gamma$ & Iters & Fails\\ 
    \midrule
    0.1 & 0.1 &  30456 & 25\\ 
    0.1 & 0.2 &  29639 & 25\\ 
    0.3 & 0.1 &  26840 & 23\\ 
    0.3 & 0.2 &  25632 & 22\\ 
    0.3 & 0.3 &  26690 & 23\\ 
    0.3 & 0.5 &  26785 & 22\\ 
    0.3 & 0.7 &  26352 & 24\\ 
    0.5 & 0.1 &  25214 & 22\\
    0.6 & 0.1 &  25114 & 21\\ 
    0.6 & 0.2 &  25722 & 23\\
    \bottomrule
\end{tabular}
\end{minipage}
\begin{minipage}[t]{.45\linewidth}
\centering
\begin{tabular}{cccc}
    \toprule
    $\eta$ & $\gamma$ & Iters & Fails\\ 
    \midrule
    0.6 & 0.3 &  25799 & 22\\ 
    0.6 & 0.5 &  27590 & 25\\ 
    0.6 & 0.7 &  25662 & 23\\ 
    0.75 & 0.3 & 26530 & 23\\ 
    0.75 & 0.5 & 26050 & 24\\ 
    0.75 & 0.7 & 28767 & 27\\ 
    \textbf{0.8} & \textbf{0.1} & \textbf{24084} & \textbf{21}\\ 
    1 & 0.3 & 26945 & 22\\ 
    1 & 0.5 & 27161 & 23\\ 
    1 & 0.7 & 28288 & 26\\ 
    \bottomrule
\end{tabular}
\end{minipage}
\end{table}

\subsubsection{Surrogate Jacobian.}
So far, all variants of the differentiable feasibility pump use the same surrogate Jacobian for the $\argmin$ of the linear relaxation: $J_{\hat{x}(\theta)} \hat{x}(\theta) = -I_n$. Yet, several methods have been proposed as discussed in \cref{sec:diff_opt}. We present experiments using the perturbation Jacobian surrogate given in Equation~\eqref{eq:pert_jacobian}, where we set the scaling parameter to $\varepsilon = 1.0$. This surrogate is based on a Monte Carlo approximation, where $M$ is the number of samples. Any positive value of $M$ leads to a significant increase in the computational time per iteration, since as many new linear problems need to be solved at each gradient computation. In our experiments, we use $M=1$, which leads to a highly noisy estimator.

We present the results of the original feasibility pump \texttt{FP} as well as \texttt{DP2} with perturbation Jacobian in \cref{tab:jac_est_results}. Even after hyperparameter tuning, the performance is significantly worse than using the simpler and faster estimator $J_{\hat{x}(\theta)} \hat{x}(\theta) = -I_n$. This suggests that the perturbation method with $M=1$ is too unstable to provide good performance in a feasibility pump context.
\begin{table}[!ht]
    \centering
    \caption{Impact of using the perturbation method for surrogate Jacobian with \texttt{FP} and \texttt{DP2}.}
    \label{tab:jac_est_results}
    \begin{tabular}{lcccc}
        \toprule
            & \multicolumn{2}{c}{\texttt{FP}}  & \multicolumn{2}{c}{\texttt{DP2}} \\
        \cmidrule(lr){2-3}\cmidrule(lr){4-5}
            & $-I_n$ & Perturbation &  $-I_n$ & Perturbation \\
        \midrule
        Number of fails            & 28     &      98   &  21       &   93    \\
        Total number of iterations & 30071 &  100974  & 24084    &  96011 \\
        \bottomrule
    \end{tabular}
\end{table}

While we investigated alternative methods such as adding a quadratic regularization term and differentiating the KKT conditions \citep{Wilder2019} or interior-point methods \citep{Mandi2020}, we could not find a method that scales successfully to the size of combinatorial problems considered in our experiments. This is because differentiable optimization typically focuses on ``easier'' optimization problems that can be repeatedly solved to optimality. In contrast, our experiments focus on large problems for which even finding a feasible solution is challenging.

\subsubsection{Gradient-descent algorithm.}
Our reinterpretation of the feasibility pump as a gradient-descent algorithm allows us to extend it to more advanced gradient-descent algorithms. In \cref{tab:gd_alg_results}, we present computational results when varying the gradient-update method to use either gradient-descent with momentum or the Adam algorithm \citep{Kingma2014adam}. The traditional gradient-descent algorithm is denoted by \texttt{GD}. The results show that changing the gradient-descent algorithm can benefit the original algorithm. However, it only slightly improves our extension \texttt{DP2} with a lower momentum value.
\begin{table}[!ht]
    \centering
    \caption{Performance with different gradient-descent algorithm}
    \label{tab:gd_alg_results}
    \begin{tabular}{lccccccc}
        \toprule
            & \multicolumn{3}{c}{\texttt{FP}}  & \multicolumn{4}{c}{\texttt{DP2}} \\
\cmidrule(lr){2-4}\cmidrule(lr){5-8}
            & GD & Mom. $0.5$ & Adam &  GD & Mom. $0.1$ & Mom. $0.5$ & Adam \\
\midrule
        Nb. fails      & 28     &  26    &   26   &   21   &   21   &   27   &   44  \\
        Nb. iters & 30071 & 27872 & 29019 & 25114 & 25020 & 29371 & 45408\\
\bottomrule
    \end{tabular}
\end{table}

\section{Conclusion}
\label{sec:conclusion}
This paper reinterprets the feasibility pump algorithm of \citet{Fischetti2005feasibility} as a gradient-descent algorithm, building upon recent works on differentiable optimization. This interpretation generalizes many existing feasibility pump algorithms and allows novel extensions in terms of designing a general loss function or gradient-update algorithms. For instance, our experiments show that using a locally convex integrality loss or a differentiable feasibility loss based on soft rounding can provide substantial performance improvements.

From a broader point of view, this paper is inscribed in the emerging framework of reinterpreting well-studied operations research methods in a differentiable optimization perspective, akin to the continuous cutting-plane generalization of \citet{Chetelat2023continuous} and its differentiable implementation by \citet{Dragotto2023differentiable}. As such, it paves the way for many promising research avenues. In fact, because so many variations of the differentiable feasibility pump are possible, only a limited array of them could be investigated in this paper. We discuss below other potential directions.

First, we kept the cycle detection and restart components identical to the ones of the initial algorithm. Investigating the combination of our work with existing variations of these components could also provide valuable insights. Another promising direction is deriving an estimator of the Jacobian of the $\argmin$ that is tailored to the differentiable feasibility pump. For instance, this could be done by taking into account the constraints of the problem. A third promising direction is adapting the proposed framework to non-linear or general integer problems. The feasibility loss introduced in this paper might prove especially relevant in that case, as it is shown to get stuck in cycles much less frequently.

\section*{Acknowledgement}
The authors thank Youssouf Emine and David E. Bernal Neira for the momentous discussions.

\bibliographystyle{abbrvnat}
\bibliography{bibliography}

\appendix
\section{Appendix: Numerical Study}
\label{app:num}

\subsection{Instances.}
\cref{tab:listinst} details the instances of the MIPLIB library used in our experiments with their number of variables and constraints.
\captionof{table}{MIPLIB instances used in experiments.} \label{tab:listinst}
\vspace{-0.3cm}
\begin{scriptsize}
\begin{longtable}{lrrrr}
    \toprule
    Name & Variables & Binary & Constraints & Equalities\\ 
    \midrule
    academictimetablesmall & 28926 & 28926 & 23294 & 1379\\ 
    air05 & 7195 & 7195 & 426 & 426\\ 
    app1-1 & 2480 & 1225 & 4926 & 1226\\ 
    app1-2 & 26871 & 13300 & 53467 & 13301\\ 
    assign1-5-8 & 156 & 130 & 161 & 31\\ 
    b1c1s1 & 3872 & 288 & 3904 & 1280\\ 
    beasleyC3 & 2500 & 1250 & 1750 & 500\\ 
    binkar101 & 2298 & 170 & 1026 & 1016\\ 
    blp-ar98 & 16021 & 15806 & 1128 & 215\\ 
    blp-ic98 & 13640 & 13550 & 717 & 90\\ 
    bnatt400 & 3600 & 3600 & 5614 & 1586\\ 
    bppc4-08 & 1456 & 1454 & 111 & 20\\ 
    cbs-cta & 24793 & 2467 & 10112 & 244\\ 
    chromaticindex1024-7 & 73728 & 73728 & 67583 & 18431\\ 
    chromaticindex512-7 & 36864 & 36864 & 33791 & 9215\\ 
    cmflsp50-24-8-8 & 16392 & 1392 & 3520 & 1096\\ 
    CMS7504 & 11697 & 7196 & 16381 & 1\\ 
    cod105 & 1024 & 1024 & 1024 & 0\\ 
    cost266-UUE & 4161 & 171 & 1446 & 1332\\ 
    csched007 & 1758 & 1457 & 351 & 301\\ 
    csched008 & 1536 & 1284 & 351 & 301\\ 
    cvs16r128-89 & 3472 & 3472 & 4633 & 0\\ 
    dano33 & 13873 & 69 & 3202 & 1224\\ 
    dano35 & 13873 & 115 & 3202 & 1224\\ 
    decomp2 & 14387 & 14387 & 10765 & 470\\ 
    drayage-100-23 & 11090 & 11025 & 4630 & 210\\ 
    drayage-25-23 & 11090 & 11025 & 4630 & 210\\ 
    dws008-01 & 11096 & 6608 & 6064 & 513\\ 
    eil33-2 & 4516 & 4516 & 32 & 32\\ 
    eilA101-2 & 65832 & 65832 & 100 & 100\\ 
    exp-1-500-5-5 & 990 & 250 & 550 & 250\\ 
    fast0507 & 63009 & 63009 & 507 & 0\\ 
    fastxgemm-n2r6s0t2 & 784 & 48 & 5998 & 144\\ 
    fhnw-binpack4-48 & 3710 & 3605 & 4480 & 0\\ 
    glass4 & 322 & 302 & 396 & 36\\ 
    glass-sc & 214 & 214 & 6119 & 0\\ 
    gmu-35-40 & 1205 & 1200 & 424 & 5\\ 
    gmu-35-50 & 1919 & 1914 & 435 & 5\\ 
    graph20-20-1rand & 2183 & 2183 & 5587 & 74\\ 
    h80x6320d & 12640 & 6320 & 6558 & 227\\ 
    irp & 20315 & 20315 & 39 & 39\\ 
    istanbul-no-cutoff & 5282 & 30 & 20346 & 221\\ 
    leo1 & 6731 & 6730 & 593 & 1\\ 
    leo2 & 11100 & 11099 & 593 & 1\\ 
    lotsize & 2985 & 1195 & 1920 & 600\\ 
    mad & 220 & 200 & 51 & 30\\ 
    markshare2 & 74 & 60 & 7 & 7\\ 
    markshare40 & 34 & 30 & 4 & 4\\ 
    mas74 & 151 & 150 & 13 & 0\\ 
    mas76 & 151 & 150 & 12 & 0\\ 
    mc11 & 3040 & 1520 & 1920 & 400\\ 
    mcsched & 1747 & 1745 & 2107 & 202\\ 
    milo-v12-6-r2-40-1 & 2688 & 840 & 5628 & 1008\\ 
    momentum1 & 5174 & 2349 & 42680 & 558\\ 
    n2seq36q & 22480 & 22480 & 2565 & 0\\ 
    n3div36 & 22120 & 22120 & 4484 & 0\\ 
    neos-1122047 & 5100 & 100 & 57791 & 0\\ 
    neos-1171448 & 4914 & 2457 & 13206 & 0\\ 
    neos-1171737 & 2340 & 1170 & 4179 & 0\\ 
    neos-1445765 & 20617 & 2150 & 2147 & 2146\\ 
    neos17 & 535 & 300 & 486 & 1\\ 
    neos-2978193-inde & 20800 & 64 & 396 & 332\\ 
    neos-2987310-joes & 27837 & 3051 & 29015 & 9342\\ 
    neos-3216931-puriri & 3555 & 3268 & 5989 & 579\\ 
    neos-3627168-kasai & 1462 & 535 & 1655 & 465\\ 
    neos-3754480-nidda & 253 & 50 & 402 & 0\\ 
    neos-4300652-rahue & 33003 & 20900 & 76992 & 6102\\ 
    neos-4387871-tavua & 4004 & 2000 & 4554 & 554\\ 
    neos-4763324-toguru & 53593 & 53592 & 106954 & 232\\ 
    neos-4954672-berkel & 1533 & 630 & 1848 & 525\\ 
    neos-5093327-huahum & 40640 & 64 & 51840 & 49792\\ 
    neos-5107597-kakapo & 3114 & 2976 & 6498 & 69\\ 
    neos-5188808-nattai & 14544 & 288 & 29452 & 5690\\ 
    neos-5195221-niemur & 14546 & 9792 & 42256 & 5690\\ 
    neos5 & 63 & 53 & 63 & 0\\ 
    neos-827175 & 32504 & 21350 & 14187 & 10512\\ 
    neos-860300 & 1385 & 1384 & 850 & 20\\ 
    neos-911970 & 888 & 840 & 107 & 35\\ 
    neos-933966 & 31762 & 27982 & 12047 & 0\\ 
    neos-957323 & 57756 & 57756 & 3757 & 0\\ 
    neos-960392 & 59376 & 59376 & 4744 & 1575\\ 
    net12 & 14115 & 1603 & 14021 & 552\\ 
    nexp-150-20-8-5 & 20115 & 17880 & 4620 & 54\\ 
    ns1116954 & 12648 & 7482 & 131991 & 235\\ 
    ns1208400 & 2883 & 2880 & 4289 & 339\\ 
    ns1830653 & 1629 & 1458 & 2932 & 565\\ 
    opm2-z10-s4 & 6250 & 6250 & 160633 & 0\\ 
    p200x1188c & 2376 & 1188 & 1388 & 200\\ 
    peg-solitaire-a3 & 4552 & 4552 & 4587 & 1367\\ 
    pg534 & 2600 & 100 & 225 & 0\\ 
    pg & 2700 & 100 & 125 & 100\\ 
    pk1 & 86 & 55 & 45 & 15\\ 
    qap10 & 4150 & 4150 & 1820 & 1820\\ 
    rail507 & 63019 & 63009 & 509 & 0\\ 
    ran14x18-disj-8 & 504 & 252 & 447 & 32\\ 
    rd-rplusc-21 & 622 & 457 & 125899 & 137\\ 
    reblock115 & 1150 & 1150 & 4735 & 0\\ 
    rmatr100-p10 & 7359 & 100 & 7260 & 1\\ 
    rmatr200-p5 & 37816 & 200 & 37617 & 1\\ 
    rocII-5-11 & 11523 & 11341 & 26897 & 781\\ 
    roi2alpha3n4 & 6816 & 6642 & 1251 & 2\\ 
    satellites2-40 & 35378 & 34324 & 20916 & 13560\\ 
    satellites2-60-fs & 35378 & 34324 & 16516 & 13560\\ 
    sct2 & 5885 & 2872 & 2151 & 360\\ 
    seymour1 & 1372 & 451 & 4944 & 0\\ 
    seymour & 1372 & 1372 & 4944 & 0\\ 
    sing326 & 55156 & 40010 & 50781 & 20825\\ 
    sing44 & 59708 & 43524 & 54745 & 22755\\ 
    sorrell3 & 1024 & 1024 & 169162 & 0\\ 
    sp150x300d & 600 & 300 & 450 & 150\\ 
    sp97ar & 14101 & 14101 & 1761 & 0\\ 
    sp98ar & 15085 & 15085 & 1435 & 0\\ 
    supportcase18 & 13410 & 13410 & 240 & 120\\ 
    supportcase26 & 436 & 396 & 870 & 0\\ 
    supportcase40 & 16440 & 2000 & 38192 & 100\\ 
    swath1 & 6805 & 2306 & 884 & 504\\ 
    swath3 & 6805 & 2706 & 884 & 504\\ 
    tbfp-network & 72747 & 72747 & 2436 & 2436\\ 
    thor50dday & 106261 & 53131 & 53360 & 230\\ 
    tr12-30 & 1080 & 360 & 750 & 360\\ 
    trento1 & 7687 & 6415 & 1265 & 1248\\ 
    uccase12 & 62529 & 9072 & 121161 & 24225\\ 
    uccase9 & 33242 & 8064 & 49565 & 11879\\ 
    uct-subprob & 2256 & 379 & 1973 & 901\\ 
    unitcal7 & 25755 & 2856 & 48939 & 2907\\ 
    \bottomrule
\end{longtable}
\end{scriptsize}

\subsection{Hyperparameter grid search.}
\cref{tab:feaslossgrid} shows the results of the hyperparameter grid search of \texttt{DP3}, which uses only the feasibility loss. \cref{tab:alllossgrid} shows the results of the hyperparameter grid search of \texttt{DP4}, which combines the integrality and feasibility losses. Finally, \cref{tab:momgrid} shows the sensitivity of using \texttt{FP} and \texttt{DP2} to the amount of momentum.
\begin{table}[!ht]
	\centering
	\caption{Sensitivity analysis of the differentiable pump \texttt{DP3} with only feasibility loss ($\beta = 0$ and $\lambda = 1$).}
	\label{tab:feaslossgrid}
	\begin{tabular}{ccccc}
		\toprule
		$\eta$ & $\gamma$ & Iters & Fails\\ 
		\midrule
		0.3 & 0.3 &    31468 & 27\\ 
		0.3 & 0.85 &    27649 & 22\\ 
		\textbf{0.3} & \textbf{1} &  \textbf{25684} & \textbf{21}\\ 
		0.75 & 0 &    47112 & 45\\ 
		0.75 & 0.3 &    26981 & 22\\ 
		0.75 & 0.85 &    28499 & 26\\ 
		0.75 & 0.85 &   28499 & 26\\ 
		0.75 & 1 &    30651 & 29\\ 
		1 & 0.3 &    27562 & 24\\ 
		1 & 0.85 &  32042 & 30\\
		1 & 1 & 41738 & 37\\ 
		\bottomrule
	\end{tabular}
\end{table}
\begin{table}[ht]
	\centering
	\caption{Sensitivity analysis of \texttt{DP4}, our most flexible differentiable feasibility pump.}
        \label{tab:alllossgrid}
        \begin{minipage}[t]{.45\linewidth}
        \small
	\begin{tabular}{ccccccc}
		\toprule
		$\eta$ & $\gamma$ & $p$ & $\lambda$ & $\beta$ & Iters & Fails\\ 
		\midrule
		0.6 & 0.1 & 2 & 1e-5 & 1e-5 & 29870 & 24\\ 
		0.6 & 0.1 & 2 & 1e-5 & 1e-3 & 27090 & 22\\ 
		0.6 & 0.1 & 2 & 1e-5 & 0.1 & 27110 & 22\\ 
		0.6 & 0.1 & 2 & 1e-5 & 1 & 25569 & 23\\ 
		0.6 & 0.1 & 2 & 1e-5 & 10 & 24717 & 22\\ 
		0.6 & 0.1 & 2 & 1e-3 & 1e-5 & 29504 & 22\\ 
		0.6 & 0.1 & 2 & 1e-3 & 1e-3 & 29192 & 23\\ 
		0.6 & 0.1 & 2 & 1e-3 & 0.1 & 26743 & 21\\ 
		0.6 & 0.1 & 2 & 1e-3 & 1 & 26435 & 23\\ 
		\textbf{0.6} & \textbf{0.1} & \textbf{2} & \textbf{1e-3} & \textbf{10} & \textbf{23383} & \textbf{20}\\ 
		0.6 & 0.1 & 2 & 0.1 & 1e-5 & 34799 & 28\\ 
		0.6 & 0.1 & 2 & 0.1 & 1e-3 & 28966 & 23\\ 
		0.6 & 0.1 & 2 & 0.1 & 0.1 & 27268 & 23\\ 
		0.6 & 0.1 & 2 & 0.1 & 1 & 26917 & 23\\ 
		0.6 & 0.1 & 2 & 0.1 & 10 & 24644 & 22\\ 
		0.6 & 0.1 & 2 & 1 & 1e-5 & 29072 & 21\\ 
		0.6 & 0.1 & 2 & 1 & 1e-3 & 27774 & 21\\ 
		0.6 & 0.1 & 2 & 1 & 0.1 & 26440 & 23\\ 
		0.6 & 0.1 & 2 & 1 & 1 & 27861 & 25\\ 
		0.6 & 0.1 & 2 & 1 & 10 & 25114 & 22\\ 
		0.6 & 0.1 & 2 & 10 & 1e-5 & 28883 & 22\\ 
		0.6 & 0.1 & 2 & 10 & 1e-3 & 30550 & 24\\ 
		0.6 & 0.1 & 2 & 10 & 0.1 & 27458 & 23\\ 
		0.6 & 0.1 & 2 & 10 & 1 & 27042 & 23\\ 
		0.6 & 0.1 & 2 & 10 & 10 & 23611 & 21\\
		\bottomrule
	\end{tabular}
        \end{minipage}\hfill
        \begin{minipage}[t]{.45\linewidth}
        \small
	\begin{tabular}{ccccccc}
		\toprule
		$\eta$ & $\gamma$ & $p$ & $\lambda$ & $\beta$ & Iters & Fails\\ 
		\midrule
		0.75 & 0.85 & 1 & 1e-5 & 1 & 33452 & 30\\ 
		0.75 & 0.85 & 1 & 1e-3 & 1 & 33250 & 30\\ 
		0.75 & 0.85 & 1 & 0.1 & 1 & 32465 & 30\\ 
		0.75 & 0.85 & 1 & 1 & 1e-5 & 27258 & 23\\ 
		0.75 & 0.85 & 1 & 1 & 1e-3 & 34490 & 32\\ 
		0.75 & 0.85 & 1 & 1 & 0.1 & 31347 & 28\\ 
		0.75 & 0.85 & 1 & 1 & 1 & 32950 & 30\\ 
		0.75 & 0.85 & 1 & 1 & 10 & 34155 & 30\\ 
		0.75 & 0.85 & 1 & 10 & 1 & 32838 & 30\\ 
		1 & 1 & 1 & 1e-5 & 1e-5 & 30727 & 27\\ 
		1 & 1 & 1 & 1e-5 & 1e-3 & 32878 & 29\\ 
		1 & 1 & 1 & 1e-5 & 0.1 & 33257 & 30\\ 
		1 & 1 & 1 & 1e-5 & 1 & 31460 & 29\\ 
		1 & 1 & 1 & 1e-5 & 10 & 31626 & 29\\ 
		1 & 1 & 1 & 1e-3 & 1e-5 & 32127 & 28\\ 
		1 & 1 & 1 & 1e-3 & 1e-3 & 33652 & 30\\ 
		1 & 1 & 1 & 1e-3 & 0.1 & 32110 & 28\\ 
		1 & 1 & 1 & 1e-3 & 1 & 34086 & 32\\ 
		1 & 1 & 1 & 1e-3 & 10 & 31740 & 29\\ 
		1 & 1 & 1 & 0.1 & 1e-5 & 32716 & 29\\ 
		1 & 1 & 1 & 0.1 & 1e-3 & 32233 & 29\\ 
		1 & 1 & 1 & 0.1 & 0.1 & 33702 & 30\\ 
		1 & 1 & 1 & 0.1 & 1 & 30762 & 28\\ 
		1 & 1 & 1 & 0.1 & 10 & 32163 & 30\\ 
		1 & 1 & 1 & 1 & 1e-5 & 28977 & 24\\ 
		1 & 1 & 1 & 1 & 1e-3 & 34225 & 32\\ 
		1 & 1 & 1 & 1 & 0.1 & 31711 & 29\\ 
		1 & 1 & 1 & 1 & 1 & 31856 & 30\\ 
		1 & 1 & 1 & 1 & 10 & 30930 & 28\\ 
		1 & 1 & 1 & 10 & 1e-5 & 27389 & 24\\ 
		1 & 1 & 1 & 10 & 1e-3 & 33713 & 31\\ 
		1 & 1 & 1 & 10 & 0.1 & 31966 & 28\\ 
		1 & 1 & 1 & 10 & 1 & 32151 & 29\\ 
		1 & 1 & 1 & 10 & 10 & 33340 & 29\\ 
		\bottomrule
	\end{tabular}
        \end{minipage}
\end{table}
\begin{table}[!ht]
    \centering
    \caption{Sensitivity analysis of \texttt{FP} and \texttt{DP2} when using gradient descent with momentum.}
    \label{tab:momgrid}
    \begin{tabular}{cccccc}
        \toprule
        $\eta$ & $\gamma$ & $p$ & Mom & Iters  & Fails \\
        \midrule
         1   & 1   & 1 & 0.1 & 30980 & 28 \\
         1   & 1   & 1 & 0.5 & 27872 & 26 \\
         1   & 1   & 1 & 0.9 & 32208 & 29 \\
         \textbf{0.6} & \textbf{0.1} & \textbf{2} & \textbf{0.1} &\textbf{25020} & \textbf{21}\\
         0.6 & 0.1 & 2 & 0.5 & 29371 & 27 \\
         0.6 & 0.1 & 2 & 0.9 & 40157 & 39 \\
        \bottomrule
    \end{tabular}
\end{table}

\end{document}